\documentclass[12pt]{amsart}

\usepackage{mystyle}

\begin{document}

\title{$p$-retract Rationality and Norm One Tori}
\author{Kazuki Sato}
\begin{abstract}
We study whether the norm one torus associated with a finite separable non-Galois field extension $K/k$ is $p$-retract rational over $k$ for a prime $p$,
focusing on the case where the Galois group of the Galois closure of $K/k$ is either the symmetric or the alternating group. 
\end{abstract}

\date{\today}
\address{Section of Liberal Arts and Sciences, National Institute of Technology (KOSEN), Ichinoseki College, Japan}
\email{kazuki-s@ichinoseki.ac.jp}
\maketitle

\section{Introduction}
In algebraic geometry, a fundamental problem is to determine whether a given algebraic variety is rational, that is, birationally equivalent to a projective space.
It is also important to determine stably rationality, retract rationality, $p$-retract rationality (for a prime $p$) and unirationality, which are weaker notions of rationality.
It is not difficult to see that 
\begin{center}
rational $\Lrarrow$ stably rational $\Lrarrow$
retract rational $\Lrarrow$
$p$-retract rational $\Lrarrow$
unirational. 
\end{center}
The notion of retract rationality was introduced by Saltman \cite{Salt84} and
that of $p$-retract rationality is due to Merkurjev \cite{Mer20}.

In this paper, we study the rationality of algebraic tori.
Let $k$ be a field and $k^{\mathrm{sep}}$ be a separable closure of $k$.
We recall that an algebraic torus over $k$ is a group scheme $T$ over $k$ that satisfies $T_{k^{\mathrm{sep}}} \cong (\mathbb{G}_{m,k^\mathrm{sep}})^n$ for some nonnegative integer $n$.
Note that an algebraic torus over $k$ is always unirational over $k$ \cite{vos}*{p. 40, Example 21}.
The rationality problem is well-understood for tori of small dimensions.
Voskresenski\u{i} \cite{vos67} showed that all tori of dimension $2$ are rational over $k$.
Kunyavski\u{i} \cite{kun90} solved the rationality problem for $3$-dimensional algebraic tori.
As a corollary, it has been proven that for $3$-dimensional algebraic tori stably rationality implies rationality  (i.e., the Zariski problem is solved affirmatively in this case).
Hoshi-Yamasaki \cite{hy17} classified algebraic tori of dimension $4$ and $5$ that are stably or retract rational.

Let $K$ be a finite separable field extension of $k$. Then we set 
\[ R^{(1)}_{K/k} \bbg_m= \Ker ( \operatorname{N}_{K/k} : \Res_{K/k} \bbg_{m,K} \to \bbg_{m,k}),\]
where $\Res_{K/k}$ is the Weil restriction.
We call it the {\it norm one torus}  associated to $K/k$.
If $K/k$ is Galois,
then 
$R^{(1)}_{K/k} \bbg_m$ is retract rational over $k$ if and only if all the Sylow subgroups of $\gal(K/k)$ are cyclic \cite{EM75},
and 
$R^{(1)}_{K/k} \bbg_m$ is $p$-retract rational over $k$ if and only if all the Sylow $p$-subgroups of $\gal(K/k)$ are cyclic \cite{Sca20}*{Proposition 5.1}.
In particular, $R^{(1)}_{K/k} \bbg_m$ is retract rational over $k$ if and only if it is $p$-retract rational over $k$ for every prime $p$.
In fact, this is true for any algebraic tori \cite{Sca20}*{Theorem 1.1}.
Hence, we consider the $p$-retract rationality of the norm one tori in the case where $K/k$ is non-Galois.

Let $S_n$ (resp. $A_n$) be the symmetric (resp. alternating) group on $n$ letters.
The main result in this paper is the following.

\begin{thm}
\label{intromain}
Let $n \geq 2$ be an integer and $p$ a prime.
Let $K/k$ be a non-Galois separable field extension of degree $n$ and $L/k$ the Galois closure of $K/k$.
\begin{enumerate}
\item Assume that $\gal (L/k)=S_n$ and $\gal(L/K)=S_{n-1}$.
Then the norm one torus 
$R^{(1)}_{K/k} \bbg_m$ associated to $K/k$
is $p$-retract rational over $k$ if and only if $n$ is a prime or $p$ is coprime to the composite $n$.
\item Assume that $\gal (L/k)=A_n$ and $\gal(L/K)=A_{n-1}$.
Then the norm one torus
$R^{(1)}_{K/k} \bbg_m$ associated to $K/k$
is $p$-retract rational over $k$ if and only if $n$ is a prime or $p$ is coprime to the composite $n$.
\end{enumerate}
\end{thm}

\begin{rem}
For the case of symmetric groups, it is known that 
the norm one torus 
$R^{(1)}_{K/k} \bbg_m$ associated to $K/k$
is retract rational over $k$ if and only if $n$ is a prime \citelist{\cite{MR878473} \cite{MR1790337} \cite{EM75}}.
For the case of alternating groups, the norm one torus
$R^{(1)}_{K/k} \bbg_m$ associated to $K/k$
is retract rational over $k$ if and only if $n$ is a prime \cite{Endo11}.
\end{rem}

We organize this paper as follows.
In Section \ref{pre}, we prepare some basic definitions and known results on the rationality of algebraic tori.
In Section \ref{pinv}, we recall the definition and establish basic properties on $p$-invertible lattices.
In Section \ref{mainsec}, we give a proof of Theorem \ref{intromain}.

\section{Preliminaries}\label{pre}
Let $k$ be a field and
$X$ be an algebraic variety over $k$.
We say that $X$ is: 
\begin{itemize}
\item \textit{rational} over $k$ if $X$ is birationally equivalent over $k$ to the affine space $\bba_k^n$;
\item \textit{stably rational} over $k$ if $X \times \bba_k^m$ is rational over $k$ for some $m \geq 0$;
\item \textit{retract rational} over $k$ if the identity map of $X$ factors rationally through some projective space;
\item {\it $p$-retract rational} over $k$ for a prime $p$ if there exists a rational dominant morphism $f:\mathbb{P}^n \dashrightarrow X$ for some $n$ such that for every nonempty open subset $U \subset \mathbb{P}^n$ in the domain of definition of $f$, there exists a morphism of varieties $g:Y \to \mathbb{P}^n$ such that $\operatorname{Im}(g) \subset U$ and the composition $f\circ g:Y \to X$ is dominant of finite degree prime to $p$: 
\begin{equation*}
\xymatrix{
{ } 
&{Y} \ar[dl]_g \ar[d]^{\mathrm{of \ degree \ prime \ to}\ p } \\
{\mathbb{P}^n} \ar@{-->}[r]^f
&{X};
}
\end{equation*}
\item \textit{unirational} over $k$ if there exists a dominant rational map $\mathbb{P}_k^n \dashrightarrow X$ for some $n \geq 0$.
\end{itemize}
The notion of retract rationality was originally introduced by Saltman \cite{Salt84} in the case where $k$ is infinite.
It has been generalized for all varieties over arbitrary fields by Merkurjev \cite{mer17}.
The notion of $p$-retract rationality is due to Merkurjev \cite{Mer20}.
It is not difficult to see that
\begin{center}
rational $\Lrarrow$ stably rational $\Lrarrow$
retract rational $\Lrarrow$
$p$-retract rational $\Lrarrow$
unirational. 
\end{center}
Note that an algebraic torus over $k$ is always unirational over $k$ \cite{vos}*{p. 40, Example 21}.
Any unirational variety is $p$-retract rational for all but finitely many primes $p$ \cite{Sca20}*{\S 1}.

Let $G$ be a finite group.
For a $G$-lattice, we mean a finitely generated $\bbz[G]$-module which is $\bbz$-free as an abelian group.
We say that a $G$-lattice $M$ is
\begin{itemize}
\item \textit{permutation} if it has a $\bbz$-basis permuted by $G$,
that is, if $M \cong \bigoplus_{1 \leq i \leq m} \bbz [G/H_i]$ for subgroups $H_1, \dots, H_m$;
\item \textit{invertible} if it is a direct summand of a permutation $G$-lattice;
\item \textit{flasque} if $\hat{H}^{-1}(H,M)=0$ for any subgroup $H$ of $G$ where $\hat{H}$ is the Tate cohomology.
\end{itemize}
Two $G$-lattices $M_1$ and $M_2$ are {\it similar} if there exist permutation $G$-lattices $P_1$ and $P_2$ such that $M_1 \oplus P_1 \cong M_2 \oplus P_2$.
We denote the similarity class of $M$ by $[M]$.
The similarity classes form a commutative monoid $S_G$ by $[M_1]+[M_2]:=[M_1 \oplus M_2]$.
It is easy to see that $M$ is an invertible $G$-lattice if and only if $[M] \in S_G$ is an invertible element.
For any $G$-lattice $M$, there is an exact sequence (a {\it flasque resolution} of $M$)
\[ 0 \to M \to P \to F \to 0 \]
of $G$-lattices, where $P$ is permutaion and $F$ is flasque \cite{CTS77}*{Lemme 3}.
The similarity class $[F]$ is uniquely determined by the similarity class $[M]$ of $M$ \cite{CTS77}*{Lemme 5}.
We denote $[F]$ by $\rho_G(M)$ and call it {\it the flasque class} of $M$.
The map $\rho_G$ from $S_G$ to itself is additive.
For a $G$-lattice $M$, it is not difficult to see that 
\[ \textrm{permutation} \Lrarrow \textrm{invertible} \Lrarrow \rho_G(M) \ \textrm{is invertible}. \]

Let $L/k$ be a finite Galois field extension and $G=\gal(L/k)$ its Galois group. 
The category of algebraic tori over $k$ which split over $L$ is dual to the category of $G$-lattices (see \cite{vos}*{p. 27, Example 6} for example).
In fact, if $T$ is an algebraic torus over $k$ which split over $L$, then the character module $\hat{T}$ of $T$ may be regarded as a $G$-lattice. 
The flasque class $\rho_G(\hat{T})$ plays a crucial role in the rationality problem for $T$.

\begin{thm}
Let $L/k$ be a finite Galois field extension and $G=\gal(L/k)$ its Galois group.
Let $T$ be an algebraic torus over $k$ which split over $L$.
\begin{enumerate}
\item $($\cite{EM73}*{Theorem 1.6}$)$
$\rho_G(\hat{T})=0$ if and only if $T$ is stably rational over $k$.
\item 
$($\citelist{\cite{Salt84}*{Theorem 3.14} \cite{Lo05}*{Lemma 9.5.4 (b)}}$)$ 
$\rho_G(\hat{T})$ is invertible if and only if $T$ is retract rational over $k$.
\end{enumerate}
\end{thm}

Let $G$ be a finite group and $H$ a subgroup of $G$.
Recall that the augmentation map
$\epsilon_{G/H} : \bbz[G/H] \to \bbz$ is defined by $\epsilon_{G/H}(gH)=1$ for every $gH \in G/H$.
Let $J_{G/H}:=(\ker \epsilon_{G/H})^\circ$ be the dual lattice of $\ker \epsilon_{G/H}$.

Let $k$ be a field and $K$ a finite separable field extension of $k$.
The norm one torus associated to $K/k$ is defined as
\[ R^{(1)}_{K/k} \bbg_m= \Ker ( \operatorname{N}_{K/k} : \Res_{K/k} \bbg_{m,K} \to \bbg_{m,k}),\]
where $\Res_{K/k}$ is the Weil restriction.
Let $L/k$ be the Galois closure of $K/k$ with $G=\gal(L/k)$ and $H=\gal(L/K)$.
Then the norm one torus 
$R^{(1)}_{K/k} \bbg_m$ associated to $K/k$
splits over $L$ and its character module is isomorphic to $J_{G/H}$ \cite{vos}*{Section 4.8}.

\begin{thm}[\citelist{\cite{EM75}*{Theorem 1.5} \cite{Salt84}*{Theorem 3.14}}]
   Let $K/k$ be a finite Galois extension with a Galois group $G$.
   Then the norm one torus $R^{(1)}_{K/k} \bbg_{m}$ is retract rational over $k$ if and only if
   all Sylow $p$-subgroups of $G$ are cyclic for all prime $p$.
\end{thm}

\begin{prop}[\cite{Endo11}*{Proposition 1.7}]
Let $K/k$ be a finite separable extension and 
$L/k$ the Galois closure of $K/k$ with $G=\gal(L/k)$ and $H=\gal(L/K)$.
Assume that 
$H$ is a nonnormal Hall subgroup of $G$.
Then the norm one torus $R^{(1)}_{K/k} \bbg_{m}$ is retract rational over $k$ if and only if
all Sylow $p$-subgroups of $G$ are cyclic for any prime $p \mid [G:H]$.
\end{prop}

\section{\texorpdfstring{$p$-invertible lattices}{p-invertible lattices}}\label{pinv}
Let $p$ be a prime, $G$ a finite group and 
$M$ a $G$-lattice.
We set $M_{(p)}: = M \otimes _\bbz \bbz_{(p)}$, where $\bbz_{(p)}$ is the localization of $\bbz$ at the prime ideal $(p)$.
We say that $M$ is $p$-{\it invertible} if there exists a permutation $G$-lattice $P$ such that $M_{(p)}$ is a direct summand of $P_{(p)}$.
It is clear that if $M$ is invertible, then it is $p$-invertible for every prime $p$.
We say that the similarity class $[M]$ of $M$ is $p$-{\it invertible} if some $E \in [M]$ is a $p$-invertible $G$-lattice, and in this case 
$E$ is $p$-invertible for any $E \in [M]$ \cite{Sca20}*{Lemma 2.3 (iv)}.
The class $[M]$ is invertible if and only if $[M]$ is $p$-invertible for every prime $p$ \cite{Sca20}*{Lemma 2.3 (ii)}.

\begin{lem}\label{lemma21}
Let $p$ be a prime,
$G$ a finite group and $H\leq G$ a subgroup.
Let $M$ be a $G$-lattice
and we view $M$ as an $H$-lattice by restriction.
If $[M]\in S_G$ is $p$-invertible, then $[M]\in S_H$ is also $p$-invertible.
In particular, if $\rho_G(M)$ is $p$-invertible, then $\rho_H(M)$ is also $p$-invertible.
\end{lem}
\begin{proof}
A permutation $G$-lattice is a permutation $H$-lattice by restrection of scalars
and the maps $\rho_G$ and $\rho_H$ are compatible with the restriction map $S_G \to S_H$ (see \cite{CTS77}*{Remarque 2}).
\end{proof}

\begin{lem}\label{lemma22}
Let $p$ be a prime.
Let $G$ be a finite group,
$G_p$ a Sylow $p$-subgroup of $G$ 
and $M$ a $G$-lattice.
The following three conditions are equivalent:
\begin{enumerate}
\item $[M]\in S_G$ is $p$-invertible;
\item $[M]\in S_{G_p}$ is $p$-invertible;
\item $[M]\in S_{G_p}$ is invertible.
\end{enumerate}
In particular, the following three conditions are equivalent:
\begin{enumerate}
\item[iv)] $\rho_G(M)$ is $p$-invertible;
\item[v)] $\rho_{G_p}(M)$ is $p$-invertible;
\item[vi)] $\rho_{G_p}(M)$ is invertible.
\end{enumerate}
\end{lem}
\begin{proof}
This is a restatement of \cite{Sca20}*{Lemma 2.5} in terms of the similarity classes.
\end{proof}

For a torus $T$ over a field $k$ which splits over a finite Galois extension $L$ of $k$ with the Galois group $G=\gal(L/k)$, 
$T$ is $p$-retract rational over $k$ if and only if the flasque class $\rho_G(\hat{T})$ is $p$-invertible \cite{Sca20}*{Proposition 3.1}.
This shows that the torus $T$ is retract rational over $k$ if and only if $T$ is $p$-retract rational over $k$ for any prime $p$ \cite{Sca20}*{Theorem 1.1}.

Let $k$ be a field and $K$ a finite separable field extension of $k$.
If $K/k$ is Galois, then 
the norm one torus $R^{(1)}_{K/k} \bbg_{m}$ associated to $K/k$ is $p$-retract rational over $k$ if and only if
a Sylow $p$-subgroup of $\gal(K/k)$ is cyclic \cite{Sca20}*{Proposition 5.1}.
When $K/k$ is non-Galois,
let $L/k$ be the Galois closure of $K/k$ with $G=\gal(L/k)$ and $H=\gal(L/K)$.
Recall that the character module of
the norm one torus 
$R^{(1)}_{K/k} \bbg_m$ associated to $K/k$
is isomorphic to $J_{G/H}$. 

\begin{prop}\label{coprime}
Let $G$ be a finite group, $H \leq G$ a subgroup of $G$ and $p$ a prime.
Assume that $p \not| \ [G:H]$ or a Sylow $p$-subgroup of $G$ is cyclic.
Then $\rho_G(J_{G/H})$ is $p$-invertible. 
\end{prop}

\begin{proof}
Let $P$ be a Sylow $p$-subgroup of $H$.
Then the short exact sequence
\[ 0 \to I_{G/H} \to \bbz[G/H] \overset{\epsilon_{G/H}}{\longrightarrow} \bbz \to 0 \]
splits as $P$-lattices, so that $J_{G/H}$ is invertible as a $P$-lattice.
If $p \not| \ [G:H]$, then since $P$ is a Sylow $p$-subgroup of $G$, we see that $\rho_G(J_{G/H})$ is $p$-invertible by Lemma \ref{lemma22}.
If a Sylow $p$-subgroup of $G$ is cyclic, then
every flasque $G$-lattice is $p$-invertible \cite{Sca20}*{Proposition 5.1}.
In particular, the flasque class $\rho_G(J_{G/H})$
is $p$-invertible.
\end{proof}

\begin{prop}
Let $G$ be a finite group, $H \leq G$ a subgroup of $G$ and $p$ a prime.
Assume that $p \mid [G:H]$ and $H \leq G$ is a Hall subgroup.
If $\rho_G(J_{G/H})$ is $p$-invertible,
then a Sylow $p$-subgroup of G is cyclic. 
\end{prop}

\begin{proof}
Let $P$ be a Sylow $p$-subgroup of $G$.
Then we have $\bbz [G/H] \cong \bbz [P]^{(t)}$ as $P$-lattices for some $t \geq 1$. 
Hence, by \cite{Endo11}*{Corollary 1.4}, $J_{G/H} \cong J_P \oplus \bbz[P]^{(t-1)}$
as $P$-lattices.
 
If $\rho_G(J_{G/H})$ is $p$-invertible, then 
$\rho_P(J_P)=\rho_P(J_{G/H})$ is invertible by Lemma \ref{lemma22}.
By \cite{EM75}*{Theorem 1.5}, $P$ is cyclic.
\end{proof}

\begin{rem}
Let $G$ be a finite group, $H \leq G$ a subgroup of $G$ and $p$ a prime.
Assume that $p \mid [G:H]$ and $H \leq G$ is not necessarily a Hall subgroup.
Using the results in \cite{hasegawa2025},
one should be able to describe a group-theoretic condition equivalent to the $p$-invertibility of $\rho_G(J_{G/H})$.
\end{rem}

\section{Proof of Theorem \ref{intromain}}\label{mainsec}
Let $S_n$ be the symmetric group on $n$ letters $\{1,2,\dots, n\}$.
We assume that the subgroup $S_{n-1}$ of $S_n$ is the stabilizer of the letter $n$ in $S_n$.

\begin{prop} \label{oddprimeS}
   Let $n$ be an integer and $p$ an odd prime.
   Assume that $n$ is not a prime and that $p$ divides $n$.
Then $\rho_{S_n}(J_{S_n / S_{n-1}})$ is not $p$-invertible.
\end{prop}

\begin{proof}
   Set $m := n/p$.
   For each $1 \leq i \leq m$, define $\rho_i : = ( (i-1)p  +1 \ (i-1)p +2 \ \cdots \ ip) \in S_n$.
   Then $P:=\lan \rho_1, \rho_2, \dots, \rho_m \ran \subset S_n$ is an elementary abelian $p$-subgroup.
   Set 
   \[P_1 := \lan \rho_2, \rho_3, \dots, \rho_m \ran, \, P_2:= \lan \rho_1, \rho_3, \dots, \rho_m \ran, \, \dots, \,
   P_m:= \lan \rho_1, \rho_2, \dots, \rho_{m-1} \ran.\]
   Then, as $P$-lattices,
   \[ \bbz [S_n/ S_{n-1}] \cong \bbz [P/P_1] \oplus \bbz [P/P_2] \oplus \dots \oplus \bbz[P/P_m].\]
   Since $m \geq 2$, $\rho_P(J_{S_n/S_{n-1}})$ is not invertible 
   by \cite{Endo01}*{Theorem 1}.
   Hence $\rho_{S_n}(J_{S_n/S_{n-1}})$ is not $p$-invertible by Lemma \ref{lemma22}.
\end{proof}

\begin{prop}\label{evenS}
Let $n \geq 6$ be an even integer.
Then $\rho_{S_n}(J_{S_n / S_{n-1}})$ is not $2$-invertible.
\end{prop}

\begin{proof}
For each $1 \leq i \leq n/2$, define $\rho_i : = ( 2i-1 \ 2i) \in S_n$.
Then $P:=\lan \rho_1, \rho_2, \dots, \rho_{n/2} \ran$ is an elementary abelian $2$-subgroup.
Set 
\[P_1 := \lan \rho_2, \rho_3, \dots, \rho_{n/2} \ran, \,P_2:= \lan \rho_1, \rho_3, \dots, \rho_{n/2} \ran, \dots, \,
P_{n/2}:= \lan \rho_1, \rho_2, \dots, \rho_{n/2-1} \ran.\]
   Then, as $P$-lattices,
   \[ \bbz [S_n/ S_{n-1}] \cong \bbz [P/P_1] \oplus \bbz [P/P_2] \oplus \dots \oplus \bbz[P/P_{n/2}].\]
Since $n/2 \geq 3$, we see that $\rho_P(J_{S_n/S_{n-1}})$ is not invertible by \cite{Endo01}*{Theorem 1}.
Therefore $\rho_{S_n}(J_{S_n/S_{n-1}})$ is not $2$-invertible by Lemma \ref{lemma22}.
\end{proof}

\begin{thm}\label{mainS}
Let $n \geq 2$ be an integer, $S_n$ the symmetric group on $n$ letters, and $p$ a prime.
Then the flasque class $\rho_{S_n}(J_{S_n / S_{n-1}})$ is $p$-invertible
if and only if 
$n$ is a prime or $p$ is coprime to the composite $n$.
\end{thm}
\begin{proof}
If $n$ is a prime, then $\rho_{S_n}(J_{S_n / S_{n-1}})$ is invertible by \cite{Endo11}*{Theorem 4.3}. So $\rho_{S_n}(J_{S_n / S_{n-1}})$ is $p$-invertible for any prime $p$.

Assume that $n$ is not a prime.
For a prime $p$ coprime to $n$, $\rho_{S_n}(J_{S_n / S_{n-1}})$ is $p$-invertible by Proposition \ref{coprime}.
For an odd prime $p | n$, $\rho_{S_n}(J_{S_n / S_{n-1}})$ is not $p$-invertible by Proposition \ref{oddprimeS}.
If $n \geq 6$ is even, then $\rho_{S_n}(J_{S_n / S_{n-1}})$ is not $2$-invertible by Proposition \ref{evenS}. If $n=4$, then $\rho_{S_n}(J_{S_n / S_{n-1}})$ is not invertible by \cite{Endo11}*{Theorem 4.3}. Hence there exists a prime $p$ such that $\rho_{S_n}(J_{S_n / S_{n-1}})$ is not $p$-invertible. The only such prime is $p=2$ by Proposition \ref{coprime}.
\end{proof}



Let $A_n$ be the alternating group on $n$ letters $\{1,2,\dots, n \}$.
We assume that the subgroup $A_{n-1}$ of $A_n$ is the stabilizer of the letter $n$ in $A_n$.

\begin{prop}\label{oddprimeA}
Let $n$ be an integer and $p$ an odd prime.
Assume that $n$ is not a prime and that $p$ divides $n$.
Then $\rho_{A_n}(J_{A_n / A_{n-1}})$ is not $p$-invertible.
\end{prop}

\begin{proof}
Set $m := n/p$.
For each $1 \leq i \leq m$, define $\rho_i : = ( (i-1)p  +1 \ (i-1)p +2 \ \cdots \ ip) \in A_n$.
Then $P:=\lan \rho_1, \rho_2, \dots, \rho_m \ran$ is an elementary abelian $p$-subgroup.
Set $P_1 := \lan \rho_2, \rho_3, \dots, \rho_m \ran, P_2:= \lan \rho_1, \rho_3, \dots, \rho_m \ran, \dots,
P_m:= \lan \rho_1, \rho_2, \dots, \rho_{m-1} \ran$.
Then, as $P$-lattices,
\[ \bbz [A_n/ A_{n-1}] \cong \bbz [P/P_1] \oplus \bbz [P/P_2] \oplus \dots \oplus \bbz[P/P_m].\]
Since $m \geq 2$, $\rho_{P}(J_{A_n/A_{n-1}})$ is not invertibley by \cite{Endo01}*{Theorem 1}.
Hence $\rho_{A_n}(J_{A_n/A_{n-1}})$ is not $p$-invertible by Lemma \ref{lemma22}.
\end{proof}

\begin{prop}\label{evenA1}
Let $n$ be an integer such that $2^2 | n$.
Then $\rho_{A_n}(J_{A_n / A_{n-1}})$ is not $2$-invertible.
\end{prop}

\begin{proof}
Define $\rho_1 = (1 \ 2)(3 \ 4) \cdot (5 \ 6)(7 \ 8) \cdot \dots \cdot (n-3 \ n-2) (n-1 \ n), \rho_2 = (1 \ 3)(2 \ 4) \cdot (5 \ 7)(6 \ 8) \cdot \dots \cdot (n-3 \ n-1) (n-2 \ n)$.
Then $P := \lan \rho_1, \rho_2 \ran$ is an elementary abelian $2$-subgroup, and as $P$-lattices, $\bbz[A_n/A_{n-1}]\cong (\bbz[P])^{n/4}$.
By \cite{Endo11}*{Corollary 1.4},
we have
\[ J_{A_n/A_{n-1}} \cong J_P \oplus (\bbz[P])^{n/4-1}. \]
Since $P$ is not cyclic, $\rho_P(J_{A_n / A_{n-1}})=\rho_P(J_P)$ is not invertible by \cite{EM75}*{Theorem 1.5}.
Hence $\rho_{A_n}(J_{A_n/A_{n-1}})$ is not $2$-invertible by Lemma \ref{lemma22}.
\end{proof}

\begin{prop}\label{evenA2}
Let $n$ be an integer.
Assume that $2 || n$ and $n \geq 6$.
Then, $\rho_{A_n}(J_{A_n / A_{n-1}})$ is not $2$-invertible.
\end{prop}

\begin{proof}
Define $\rho_1= (1 \ 2) (3 \ 4)$, $\rho_2 = (1 \ 2)(5 \ 6) (7 \ 8)  \cdots (n-1 \ n)$, $\rho_3 = \rho_1 \rho_2 = (3 \ 4) (5 \ 6) \cdots (n-1 \ n)$.
Then $P:=\lan \rho_1, \rho_2 \ran = \{ e,\rho_1, \rho_2, \rho_3 \} \subset A_n$ is an elementary abelian $2$-subgroup, and as $P$-lattices,
\[ \bbz[A_n/A_{n-1}] \cong \bbz[P/\lan \rho_1 \ran]^{(n/2-2)} \oplus \bbz[P/\lan \rho_2 \ran] \oplus \bbz[P/\lan \rho_3 \ran]. \]
By \cite{Endo01}*{Theorem 1}, we see that $\rho_P(J_{A_n/A_{n-1}})$ is not invertible.
Hence $\rho_{A_n}(J_{A_n/A_{n-1}})$ is not $2$-invertible by Lemma \ref{lemma22}.
\end{proof}

\begin{thm}\label{mainA}
Let $n\geq 2$ be an integer, $A_n$ the alternating group on $n$ letters, and $p$a prime.
Then, 
the flasque class $\rho_{A_n}(J_{A_n / A_{n-1}})$ is $p$-invertible
if and only if 
$n$ is a prime or $p$ is coprime to the composite $n$.
\end{thm}

\begin{proof}
This follows from Propositions \ref{oddprimeA}, \ref{evenA1}, and \ref{evenA2} as in the proof of Theorem \ref{mainS}. 
\end{proof}

\begin{proof}[Proof of Theorem \ref{intromain}]
The norm one torus $R^{(1)}_{K/k} \bbg_m$ associated to $K/k$ is $p$-retract rational over $k$ if and only if the flasque class of its character module is $p$-invertible \cite{Sca20}*{Proposition 3.1}.
The result now follows from Theorems \ref{mainS} and \ref{mainA}.
\end{proof}

\bibliography{myrefs}

\end{document}